\documentclass[12pt]{amsart}
\usepackage{amssymb,latexsym,comment,url}
\usepackage[all]{xy}

\newfont{\wncyr}{wncyr10 at 12pt}
\newfont{\wncyrten}{wncyr10 at 10pt}

\newcommand{\Br}{\operatorname{Br}}
\newcommand{\Gal}{{\operatorname{Gal}}}
\newcommand{\ord}{{\operatorname{ord}}}
\newcommand{\CT}{{\operatorname{CT}}}
\newcommand{\SL}{{\operatorname{SL}}}

\newcommand{\eps}{\varepsilon}
\newcommand{\Jac}{{\operatorname{Jac}}}
\newcommand{\disc}{\operatorname{disc}}
\newcommand{\content}{\operatorname{content}}
\newcommand{\im}{\operatorname{im}}
\newcommand{\Pic}{\operatorname{Pic}}
\newcommand{\Div}{\operatorname{Div}}
\newcommand{\divv}{\operatorname{div}}

\newcommand{\pr}{\operatorname{pr}}
\newcommand{\rank}{\operatorname{rank}}
\newcommand{\inv}{\operatorname{inv}}
\newcommand{\ra}{\longrightarrow}
\newcommand{\OO}{{\mathcal O}}
\newcommand{\F}{{\mathbb F}}
\newcommand{\PP}{{\mathbb P}}
\newcommand{\Kbar}{{\overline K}}
\newcommand{\Q}{{\mathbb Q}}
\newcommand{\Gm}{{\mathbb G}_m}
\newcommand{\Z}{{\mathbb Z}}
\newcommand{\isom}{\cong}
\newcommand{\Sha}{\mbox{\wncyr Sh}}

\newenvironment{ProofOf}[1]{\par\noindent{\em Proof of #1.}}%
                        {\hspace*{\fill}\nobreak$\Box$\par\medskip}

\newtheorem{Proposition}{Proposition}[section]
\newtheorem{Theorem}[Proposition]{Theorem}
\newtheorem{Lemma}[Proposition]{Lemma}
\newtheorem{Corollary}[Proposition]{Corollary}

\theoremstyle{definition}

\newtheorem{Remark}[Proposition]{Remark}

\newtheorem{Example}[Proposition]{Example}

\begin{document}

\title[Binary quartics]%
{On binary quartics and the Cassels-Tate pairing}

\author{Tom~Fisher}
\address{University of Cambridge,
         DPMMS, Centre for Mathematical Sciences,
         Wilberforce Road, Cambridge CB3 0WB, UK}
\email{T.A.Fisher@dpmms.cam.ac.uk}

\date{31st August 2022}  

\begin{abstract}
We use the invariant theory of binary quartics to give a new
formula for the Cassels-Tate pairing on the $2$-Selmer group
of an elliptic curve. Unlike earlier methods, our formula does
not require us to solve any conics. An important role in 
our construction is played by a certain $K3$ surface defined by
a $(2,2,2)$-form.
\end{abstract}

\maketitle

\renewcommand{\baselinestretch}{1.1}
\renewcommand{\arraystretch}{1.3}
\renewcommand{\theenumi}{\roman{enumi}}

\section{Introduction}

Let $E$ be an elliptic curve over a number field $K$. The Mordell-Weil
theorem tells us that the abelian group $E(K)$ is finitely generated,
but there is no known algorithm guaranteed to compute its rank.
Instead, for each integer $n \ge 2$ there is an exact sequence of 
abelian groups
\[ 0 \to E(K)/n E(K) \to S^{(n)}(E/K) \to \Sha(E/K)[n] \to 0. \]
The $n$-Selmer group $S^{(n)}(E/K)$ is finite and effectively computable.
Computing $S^{(n)}(E/K)$ gives an upper bound for the rank of $E(K)$,
but this will be sharp only if the $n$-torsion of the 
Tate-Shafarevich group $\Sha(E/K)$ is trivial.

Cassels \cite{CaIV} showed that there is an alternating pairing
\[ \langle~.~\rangle_{\CT} : S^{(n)}(E/K) \times S^{(n)}(E/K) \to \Q/\Z \]
whose kernel is the image of $S^{(n^2)}(E/K)$. 
By computing this pairing,
our upper bound for the rank of $E(K)$ improves from that
obtained by $n$-descent to that obtained by $n^2$-descent.
In view of the generalisation to abelian varieties, due to Tate,
the pairing is known as the Cassels-Tate pairing.
 
Cassels \cite{Ca98} also described a method for computing the pairing
in the case $n=2$. His method involves solving conics over the field
of definition of each $2$-torsion point on $E$. More recently, Donnelly
\cite{D} 
found a method that only involves solving conics over $K$, and implemented
this in {\tt Magma} \cite{magma}. In this
article we use the invariant theory of binary quartics to give a
self-contained account of a version
of his method that is relatively simple to implement.

\enlargethispage{3ex}

Since this article was first written, Jiali Yan has written her
PhD thesis \cite{Yan}, extending some of these ideas
to Jacobians of genus 2 curves, and Bill Allombert has implemented
our method for computing the pairing as part of the function {\tt ellrank}
in {\tt pari/gp} \cite{pari}.
I thank them both, and also Steve Donnelly and John Cremona, for
useful discussions.

\section{Binary quartics}
\label{sec:bq}

A {\em binary quartic} over a field $K$ is a homogeneous polynomial
$g \in K[x,z]$ of degree $4$. Binary quartics 
$g_1$ and $g_2$ are {\em $K$-equivalent}
if 
\[ g_2(x,z) = \lambda^2 g_1 (\alpha x + \gamma z, \beta x + \delta z) \]
for some $\lambda, \alpha, \beta, \gamma, \delta \in K$ with 
$\lambda(\alpha \delta - \beta \gamma) \not= 0$. They are {\em properly}
$K$-equivalent if in addition $\lambda(\alpha \delta - \beta \gamma) = \pm 1$.
The invariants of the binary quartic
\begin{equation}
\label{bq}
g(x,z) = a x^4 + b x^3 z + c x^2 z^2 + d x z^3 + e z^4
\end{equation}
are 
\begin{align*}
I &= 12 a e - 3 b d + c^2 , \\
J &= 72 a c e - 27 a d^2 - 27 b^2 e + 9 b c d - 2 c^3.
\end{align*}
The binary quartics $g_1$ and $g_2$ have invariants related by
$I(g_2) = \lambda^4 (\alpha \delta - \beta \gamma)^4 I(g_1)$ and
$J(g_2) = \lambda^6 (\alpha \delta - \beta \gamma)^6 J(g_1)$.
In particular, properly equivalent binary quartics have
the same invariants.
The discriminant is $\Delta = 16(4 I^3 - J^2)/27$.
We say that $g$ is {\em $K$-soluble} if there exist $x,z \in K$, not 
both zero, such that $g(x,z)$ is a square in $K$. The reason for
this terminology is that if $\Delta(g) \not=0$ then there is 
a smooth projective curve $C$ of genus one with affine equation 
$y^2 = g(x,1)$, and we are asking that $C(K) \not= \emptyset$.
As shown by Weil~\cite{Weil54}, the Jacobian of $C$ is the elliptic curve
\begin{equation}
\label{jac}
 E_{I,J} : \quad y^2 = x^3 - 27 I x - 27 J. 
\end{equation}

Now let $K$ be a number field, and $M_K$ its set of 
places. A binary quartic over $K$ is
{\em everywhere locally soluble} if it is $K_v$-soluble for all places
$v \in M_K$. We note that every elliptic curve over $K$ can be
written in the form~\eqref{jac} for some $I,J \in K$ with 
$4 I^3 - J^2 \not =0$.

\begin{Lemma}
\label{propsel2}
Let $I,J \in K$ with $4 I^3 - J^2 \not =0$. Then
\[ S^{(2)}(E_{I,J}/K) = \left\{ \begin{array}{c}
\text{everywhere locally soluble} \\
\text{binary quartics over $K$} \\
\text{with invariants $I$ and $J$} \end{array} \right\}
/(\text{proper $K$-equivalence}). \]
\end{Lemma}
\begin{proof}
The case $K = \Q$ is proved in \cite{BSDI}, the only simplification
in this case being that (since the only roots of unity in $\Q$ are
$\pm 1$) equivalent quartics with the same invariants
are always properly equivalent (even in the cases where $I=0$ or $J=0$).
The general case is similar.
\end{proof}

Although Lemma~\ref{propsel2} specifies $S^{(2)}(E_{I,J}/K)$ 
as a set, the group law is not obvious. The following description
is taken from \cite{Cr}, \cite{CF}. Let $L$ be the \'etale algebra
$K[\varphi]$ where
$\varphi$ is a root of $X^3 - 3 I X + J = 0$. Then the binary
quartic~\eqref{bq} has {\em cubic invariant}
\[ z(g) = \frac{4 a \varphi + 3 b^2 - 8 a c}{3}. \]
By a change of coordinates (that is, replacing $g$ by a properly
equivalent quartic) we may assume that $z(g)$ is a unit in $L$.
The group law on $S^{(2)}(E_{I,J}/K)$ is then given by multiplying
the cubic invariants in $L^\times/(L^\times)^2$. The 
method for converting an element of $L^\times/(L^\times)^2$ back to a binary
quartic does, however, involve solving a conic over $K$.

\section{Statement of results}
\label{sec:stat}

In this section we state our new formula for the Cassels-Tate pairing
on the $2$-Selmer group of an elliptic curve.
First we need some more invariant theory.
The binary quartic~\eqref{bq} has Hessian 
\begin{align*}
 h (x, z ) = (3&b^2 - 8ac)x^4 + 4(bc - 6ad)x^3 z
 + 2(2c^2 - 24ae - 3bd)x^2 z^2 \\ & +4(cd - 6be)xz^3 + (3d^2 - 8ce)z^4. 
\end{align*}
There are exactly three linear combinations of 
$g(x,z)$ and $h(x,z)$ that are singular (i.e. have repeated roots).
Following~\cite{CF} this prompts us to put 
\begin{align} 
G(x,z) &= \frac{1}{3} \left( 4 \varphi g(x,z) + h(x,z) \right), \\
\label{def:H}
H(x,z) &= \frac{1}{12} \frac{\partial^2 G}{\partial x^2} + \frac{2}{9} (I 
-\varphi^2) z^2,
\end{align}
so that $G(1,0) G(x,z) = H(x,z)^2$. We note that $z(g) = G(1,0) = H(1,0)$.

\begin{Theorem}
\label{thm:main}
Let $I,J \in K$ with $4 I^3 - J^2 \not =0$.
Let $g_1,g_2,g_3$ be 
everywhere locally soluble binary quartics over $K$ 
with invariants $I$ and $J$.
Let $H_1(x,z)$ be the binary quadratic
form~\eqref{def:H}, with coefficients in $L = K[\varphi]$, 
associated to $g_1$. Suppose that $z(g_1) z(g_2) z(g_3) = m^2$
for some $m \in L^\times$, and write
\begin{equation}
\label{getgamma}
  \frac{z(g_2) z(g_3)}{m} H_1(x,z) = \alpha_1(x,z) + \beta_1(x,z) 
\varphi + \gamma_1(x,z) \varphi^2 
\end{equation}
where $\alpha_1,\beta_1,\gamma_1 \in K[x,z]$. For each $v \in M_K$ 
we choose $x_v,z_v \in K_v$ with 
$g_1(x_v,z_v)$ a square in $K_v$ and $\gamma_1(x_v,z_v) \not= 0$. 
If $g_2(1,0) \not=0$ then 
the Cassels-Tate pairing on $S^{(2)}(E_{I,J}/K)$ is given by
\begin{equation}
\label{myctp}
 \langle [g_1], [g_2] \rangle_{\CT} = \prod_{v \in M_K} 
(g_2(1,0),\gamma_1(x_v,z_v))_v 
\end{equation}
where $(~,~)_v : K_v^\times/(K_v^\times)^2 \times K_v^\times/(K_v^\times)^2
\to \mu_2$ is the Hilbert norm residue symbol.
\end{Theorem}

\begin{Remark}
\label{rems}
(i) If we wish to compute the pairing starting
only with $g_1$ and $g_2$, then we first change coordinates so 
that $z(g_1)$ and $z(g_2)$ are units in $L$, multiply these together, 
and then compute $g_3$ by solving a conic over $K$.
This conic is the same as the one that has to be solved 
in Donnelly's method \cite{D}. \\
(ii) We show in Remark~\ref{rem:gamma} 
that the binary quadratic form $\gamma_1$ is
not identically zero. Therefore, by our assumption that $g_1$ is 
everywhere locally soluble, it is always possible to choose 
$x_v, z_v \in K_v$ with the stated properties. \\
(iii) The assumption that $g_2(1,0) \not=0$ is no limitation, since 
if $g_2(1,0)=0$ then $[g_2] = 0$ in the $2$-Selmer group, which
certainly implies the pairing is trivial. \\
(iv) By definition the Cassels-Tate pairing takes values in $\Q/\Z$.
In our formula it takes values in $\mu_2$. It should be understood
that we have identified $\mu_2 = \frac{1}{2}\Z/\Z$. \\
(v) Since $\langle~,~\rangle_{\CT}$ is alternating and bilinear
and $[g_1] + [g_2] + [g_3] = 0$ we have
$\langle [g_1], [g_2] \rangle_{\CT} = \langle [g_1], [g_3] \rangle_{\CT}$.
So we may equally write $g_2(1,0)$ or $g_3(1,0)$ in~\eqref{myctp}.
Notice however that the binary quartics $g_2$ and $g_3$ do both 
contribute to the pairing via~\eqref{getgamma}. 
Moreover we must use the exact formulae for 
$z(g_1)$, $z(g_2)$ and $z(g_3)$, these being linear in $\varphi$.
It is not enough just to 
know these quantities up to squares, since this would change
the left hand side of~\eqref{getgamma}. \\
(vi) If $E(K)[2]=0$ then $m$ is uniquely determined up to sign.
By the product formula for the Hilbert norm residue symbol this makes
no difference to~\eqref{myctp}. If $E(K)[2] \not=0$ then there are
more choices for $m$, but it turns out (see the proof of
Theorem~\ref{thm:m}) that we may use any one of these to
compute the pairing.
\end{Remark}

\enlargethispage{2.5ex}

\begin{Remark}
The product over all places in Theorem~\ref{thm:main} is a finite product.
Indeed, outside an easily determined finite set of places, we have
\begin{enumerate}
\item $v$ is a finite prime, with residue field of size at least $11$.
\item $g_1$ and $\gamma_1$ have $v$-adically integral coefficients,
  with $v \nmid \Delta(g_1) \content(\gamma_1)$.
\item $g_2(1,0)$ is a $v$-adic unit and $v \nmid 2$.
\end{enumerate}
Under conditions (i) and (ii) we can pick our local point
(by Hensel lifting a smooth point on the reduction that is not
a root of $\gamma_1$) such that
$\gamma_1(x_v,z_v)$ is a unit. It follows by (iii) that
the local contribution at $v$ 
is trivial.
\end{Remark}

\begin{Example} Let $E/\Q$ be the elliptic curve
\[  y^2 + y = x^3 - x^2 - 929 x - 10595 \]
labelled $571a1$ in \cite{CrTables}.
A $2$-descent shows that $S^{(2)}(E/\Q) \isom (\Z/2\Z)^2$, and its
non-zero elements are represented by
\begin{align*}
g_1(x,z) &= -11 x^4 + 68 x^3 z - 52 x^2 z^2 - 164 x z^3 - 64 z^4, \\
g_2(x,z) &= -4 x^4 - 60 x^3 z - 232 x^2 z^2 - 52 x z^3 - 3 z^4, \\
g_3(x,z) &= -31 x^4 - 78 x^3 z + 32 x^2 z^2 + 102 x z^3 - 53 z^4.
\end{align*}
Each of these binary quartics has invariants $I=44608$ and $J=18842960$, 
and discriminant $\Delta = -2^{12} \cdot 571$. 
By~\eqref{getgamma}, with 
$m = \frac{1}{9}(20 \varphi^2 - 8656 \varphi + 936032)$,
we get \[\gamma_1(x,z) = \frac{4}{9} (5 x^2 - 16 x z - 12 z^2).\]

For each odd prime $p$ there is a smooth $\F_p$-point
on the reduction of $y^2 = g_1(x,1)$ mod $p$, whose $x$-coordinate 
is not a root of $5 x^2 - 16 x - 12 = 0$.
Indeed we checked this claim directly for $p=3,5,7,11$ and $571$, and for
all other primes it follows by Hasse's bound. Therefore the odd primes
make no contribution to~\eqref{myctp}.

To compute the contribution at $p=2$ we write $g_1(x,1) = x^4 + 4 q(x)$
where
\[  q(x) = -3 x^4 + 17 x^3 - 13 x^2 - 41 x  - 16. \]
By Hensel's lemma the equation $q(x)=0$ has a root in $\Z_2$ 
with $x = 2^4 + O(2^5)$.
But then $\gamma_1(x,1) \equiv 5 \mod (\Q_2^\times)^2$, and since 
$(5,-1)_2=1$ the contribution is again trivial. 
Finally, since $g_1(15,4) >0$ and $\gamma_1(15,4) < 0$,
there is a contribution from the real place. This shows that 
the Cassels-Tate pairing on $S^{(2)}(E/\Q)$ is non-trivial, 
and hence $\rank E (\Q) = 0$.
\end{Example}

\section{The Cassels-Tate pairing}
\label{sec:ctp}

There are two standard definitions of the Cassels-Tate pairing
(in the case of elliptic curves) called in~\cite{F}, \cite{PS} the 
{\em homogeneous space definition} and the {\em Weil pairing
definition}. Both definitions appear in Cassels' original paper
\cite{CaIV}, although the method in \cite{Ca98} (see also \cite{FSS})
is a variant of the Weil pairing definition. 
In this section we
review the homogeneous space definition, and highlight its connection
with the Brauer-Manin obstruction.

Let $K$ be a field with separable closure $\Kbar$. We write
$H^i(K,-)$ for the Galois cohomology group $H^i(\Gal(\Kbar/K),-)$.
Let $C/K$ be a smooth projective curve. We define
\begin{equation}
\label{def:BrC}
\Br(C) = \ker \bigg( H^2(K,\Kbar(C)^\times) \to H^2(K, \Div C) \bigg).
\end{equation}
It is shown in the Appendix to \cite{Lichtenbaum} that
this is equivalent to the usual definition $\Br(C) = H^2_{{\text{\'et}}}(C,\Gm)$.
Identifying $\Br(K) = H^2(K,\Kbar^\times)$, there is a natural map
\begin{equation}
\label{incl}
\Br(K) \to \Br(C). 
\end{equation}
We will need the following two facts, whose proofs we give below.

(i) For $P \in C(K)$ there is an evaluation map
\[  \Br(C) \to \Br(K) \, ; \,\, A \mapsto A(P). \]
This is a group homomorphism, and a section to the map~\eqref{incl}.
Moreover the evaluation maps behave functorially with respect to 
all field extensions.

(ii) Suppose $C$ is a smooth curve of genus one, with Jacobian
elliptic curve $E$. If $H^3(K,\Kbar^\times) = 0$
then there is an isomorphism 
\begin{equation}
\label{PsiC}
 \Psi_C : \frac{H^1(K,E)}{\langle[C]\rangle} \stackrel{\sim}{\ra} 
\frac{\Br(C)}{\Br(K)}. 
\end{equation}

Now let $E$ be an elliptic curve over a number field $K$. 
Let $C$ and $D$ be principal homogeneous spaces under $E$. 
Since $H^3(K,\Kbar^\times) = 0$ 
for $K$ a number field, we have $\Psi_C([D]) = A \mod{\Br(K)}$
for some $A \in \Br(C)$.
Now suppose that $C$ and~$D$ are everywhere locally soluble. 
For each place $v \in M_K$ we pick a local point $P_v \in C(K_v)$.
The Cassels-Tate pairing 
$\Sha(E/K) \times \Sha(E/K) \to \Q/\Z$ is defined by
\begin{equation}
\label{def:ctp}
 \langle [C], [D] \rangle_{\CT} = \sum_{v \in M_K} \inv_v (A(P_v)) 
\end{equation}
where $\inv_v : \Br(K_v) \to \Q/\Z$ is the local invariant map.  As
this form of the definition makes clear, if $\langle [C], [D]
\rangle_{\CT} \not=0$ then the genus one curve $C$ is a 
counter-example to the Hasse Principle
explained by the Brauer-Manin obstruction.

We check that the pairing is well defined, i.e. it does not depend 
on the choices of $A$ and of the $P_v$. By class field theory there 
is an exact sequence
\[ 0 \ra \Br(K) \ra \bigoplus_{v \in M_K} \Br(K_v) 
\stackrel{\sum \inv_v}{\ra} \Q/\Z \ra 0. \]
It follows that if we change $A$ by adding an element of $\Br(K)$ then the 
pairing~\eqref{def:ctp} is unchanged. Next, since the class of $D$ 
is trivial in $H^1(K_v,E)$, the analogue of~\eqref{PsiC} over $K_v$ shows that
the restriction of $A$ to the Brauer group of $C/K_v$ is constant, 
i.e. it belongs to 
the image of $\Br(K_v)$. Therefore the pairing~\eqref{def:ctp}
does not depend on the choice of local points $P_v$.

We now prove the facts we quoted in (i) and (ii) above.

(i) For $P \in C(K)$ there is a short exact sequence of Galois modules
\[ 0 \ra \OO_P^\times \ra \Kbar(C)^\times \stackrel{\ord_P}{\ra} \Z \ra 0 \]
where $\OO_P$ is the local ring at $P$. Taking Galois cohomology gives
an exact sequence 
\[ 0 \ra H^2(K,\OO_P^\times) \ra H^2(K,\Kbar(C)^\times) 
\stackrel{\ord_P}{\ra}  H^2(K,\Z). \]
It follows by~\eqref{def:BrC} 
that each element of $\Br(C)$ can be represented by
a cocycle taking values in $\OO_P^\times$, and so can be evaluated at $P$.

(ii) There is an exact sequence of Galois modules
\[ 0 \to \Kbar^\times \to \Kbar(C)^\times \to \Div C \to \Pic C \to 0, \]
where $\Div C$ and $\Pic C$ are the divisor group and Picard group
for $C$ over $\Kbar$.
Splitting into short exact sequences, and taking Galois cohomology,
gives the following exact sequences
\[ \xymatrix{ & & H^2(K,\Kbar^\times) \ar[d] \\ 
& & H^2(K,\Kbar(C)^\times ) \ar[d] \\
H^1(K, \Div C) \ar[r] & H^1(K,\Pic C) \ar[r] & H^2(K,\Kbar(C)^\times
/\Kbar^\times ) \ar[r] \ar[d] & H^2(K, \Div C) \\ 
& & H^3(K,\Kbar^\times) }
\]
By Shapiro's lemma and the fact that $H^1(K,\Z)=0$ we have $H^1(K, \Div C)=0$.
It follows by~\eqref{def:BrC} and a diagram chase that
there is an exact sequence
\begin{equation}
\label{hs}
 \Br(K) \to \Br(C) \to H^1(K, \Pic C) \to H^3(K,\Kbar^\times).
\end{equation}
In fact, had we started from the definition
$\Br(C) = H^2_{{\text{\'et}}}(C,\Gm)$,
then~\eqref{hs} would follow
from the Hochschild--Serre spectral sequence.

If $C$ is a smooth curve of genus one with Jacobian $E$, then taking
Galois cohomology of the exact sequence
\[ 0 \ra \Pic^0 C \ra \Pic C \stackrel{\deg}{\ra} \Z \ra 0 \]
gives
\begin{equation}
\label{exseq}
\Z  \stackrel{\delta}{\ra} H^1(K,E) \ra H^1(K, \Pic C) \ra 0 
\end{equation}
with $\delta(1) = [C]$.
If $H^3(K,\Kbar^\times) = 0$ then from~\eqref{hs} and~\eqref{exseq}
we obtain the isomorphism $\Psi_C$.

\section{Cyclic extensions}

The definition of the map $\Psi_C$ in the last section simplifies
when we evaluate it on classes split by a cyclic extension $L/K$.
Let $G = \Gal(L/K)$ be generated by $\sigma$ of order $n$. We recall that
for $A$ a $G$-module, the Tate cohomology groups are
\[ \widehat{H}^0(G,A) = \frac{\ker(\Delta | A)}{\im(N | A)}
\quad \text{ and } \quad
    \widehat{H}^1(G,A) = \frac{\ker(N | A)}{\im(\Delta | A)} \]
where $\Delta = 1 - \sigma$ and $N = 1 + \sigma + \ldots + \sigma^{n-1}$
satisfy $\Delta N = N \Delta = 0$ in $\Z[G]$.

For $b \in K^\times$ there is a cyclic $K$-algebra with basis 
$1,v, \ldots, v^{n-1}$ as an $L$-vector space, and multiplication 
determined by $v^n = b$ and $v x = \sigma(x) v$ for all $x \in L$.
We write $(L/K,b)$ for the class of this algebra in
$\Br(K)=H^2(K,\Kbar^\times)$.
Likewise if $f \in K(C)^\times$ then $(L/K,f)$ is an element of
$H^2(K,\Kbar(C)^\times )$.
\begin{Lemma}
\label{lem:cyc}
Suppose $\Xi \in \Div_L^0 C$ with $N_{L/K}(\Xi) =  \divv (f)$ for some
$f \in K(C)^\times$.
If $\xi$ is the image of $\Xi$ under 
\[ \widehat{H}^1(G,\Pic_L^0 C ) \isom H^1(G,\Pic_L^0 C )
  \stackrel{\rm inf}{\ra} H^1(K,E) \]
then $\Psi_C(\xi) = (L/K,f)$.
\end{Lemma}
\begin{proof}
We follow the construction of $\Psi_C$ in Section~\ref{sec:ctp}.
We start with the exact sequence of $G$-modules
\[ 0 \to L^\times \to L(C)^\times \to \Div_L C \to \Pic_L C \to 0. \]
Splitting into short exact sequences, and taking Galois cohomology,
gives a diagram as before. The connecting map
\[ \widehat{H}^1(G,\Pic_L C) 
    \to \widehat{H}^0(G,L(C)^\times/L^\times )  \]
is now given by $\Xi \mapsto f$. Therefore
$\Psi_C(\xi)$ is the image of $f$ under the map
\[ \frac{K(C)^\times}{N_{L/K}(L(C)^\times)}
 = \widehat{H}^0(G,L(C)^\times ) \isom H^2(G,L(C)^\times )
  \stackrel{\rm inf}{\ra} H^2(K,\Kbar(C)^\times ). \]
This is the cyclic algebra $(L/K,f)$ as required.
\end{proof}

\section{Pairs of binary quartics and $(2,2)$-forms}

Let $C$ be a smooth curve of genus one. 
First suppose, as in Section~\ref{sec:bq}, that
$C$ is defined by a binary
quartic $g$. 
Then $C \to \PP^1$ is a double cover ramified over the $4$ roots of $g$. 
We write $H$ for the hyperplane
section (i.e., fibre of the map $C \to \PP^1$), and $\iota$ for 
the involution on $C$ with $Q + \iota(Q) \sim H$ for all $Q \in C$.

Next we suppose that $C \subset \PP^1 \times \PP^1$ 
is defined by a $(2,2)$-form, i.e., a polynomial 
$f(x_1,z_1;x_2,z_2)$ that is 
homogeneous of degree $2$ in each of the sets of variables $x_1$,$z_1$ 
and $x_2$,$z_2$. Projecting
$C$ to either factor gives a double cover of $\PP^1$. 
The corresponding binary quartics are obtained 
by writing $f$ as a binary
quadratic form in one of the sets of variables, and taking its 
discriminant. We write $\pr_1, \pr_2 : C \to \PP^1$ for the
projection maps. Let $H_1, H_2$ and $\iota_1, \iota_2$
be the corresponding hyperplane sections and involutions.

\begin{Lemma} 
\label{lem:22}
Let $C = \{ f(x_1,z_1;x_2,z_2) = 0 \} \subset \PP^1 
\times \PP^1$ as above.
\begin{enumerate}
\item The composite $\iota_1 \iota_2$ is translation by 
some $P \in E = \Jac(C)$. Moreover the isomorphism
$\Pic^0(C) \isom E$ sends $[H_1 - H_2] \mapsto P$.
\item If $H_i = \pr_i^*(1:0)$ then 
\[ \divv ( f(x_1,z_1;1,0)/z_1^2 ) = H_2 + \iota_1^*H_2 - 2 H_1. \]
\end{enumerate}
\end{Lemma}
\begin{proof} 
(i) If $Q \in C$ then $\iota_2 Q + \iota_1 \iota_2 Q \sim H_1$
and $Q + \iota_2 Q \sim H_2$. Subtracting one from the other gives 
$[\iota_1 \iota_2 Q - Q ] = [H_1 - H_2]$ as required. \\
(ii) The specified rational function on $C$ factors via $\pr_1$
and is therefore invariant under pull back by $\iota_1$. It has
a zero at each point in the support of $H_2$, and a double pole
at each point in the support of $H_1$. Since there are no other
poles, and the divisor has degree $0$, it must therefore be as
stated.
\end{proof} 

\begin{Remark}
Lemma~\ref{lem:22}(i) is closely related to Poncelet's Porism, as 
described in \cite{GH}. Our use of $(2,2)$-forms was inspired by
the treatment in \cite{BH}.
\end{Remark}
We write $\disc_k(f)$ for the discriminant of 
$f$ when it is viewed as a binary quadratic form
in the $k$th set of variables.

\begin{Lemma}
\label{prop:22}
Let $C = \{ f(x_1,z_1;x_2,z_2) = 0 \} \subset \PP^1 
\times \PP^1$ as above.
Let $a \in K^\times$ and let $C_1$, $C_2$ be the following quadratic
twists of $C$.
\begin{align*}
C_1: \qquad a y^2 &= \disc_2(f) \\
C_2: \qquad a y^2 &= \disc_1(f) 
\end{align*}
Then $\Psi_{C_1}([C_2]) = (K(\sqrt{a})/K, f(x_1,z_1;1,0)/z_1^2)$.
\end{Lemma}
\begin{proof}
If $a \in (K^\times)^2$ then $C_1$ and $C_2$ are isomorphic over
$K$ and so by~\eqref{PsiC} we have $\Psi_{C_1}([C_2]) = 0$.
We may therefore suppose that $a \not\in (K^\times)^2$.
Let $L = K(\sqrt{a})$ and $G = \Gal(L/K) = \{1, \sigma\}$. 
We claim there is a divisor $\Xi \in \Div_L^0(C_1)$ such that
\begin{enumerate}
\item $C_2$ is the twist of $C_1$ by the class of $\Xi$ in 
$\widehat{H}^1(G,\Pic_L^0(C_1))$, and
\item $N_{L/K} (\Xi) = \divv (f(x_1,z_1;1,0)/z_1^2)$.
\end{enumerate}
Then by Lemma~\ref{lem:cyc} we have
$\Psi_{C_1}([C_2] - [C_1]) = (L/K, f(x_1,z_1;1,0)/z_1^2)$. Since
$\Psi_{C_1}$ is a group homomorphism and $\Psi_{C_1}([C_1]) = 0$ this 
proves the lemma.

We construct $\Xi$ as follows. We factor the projection map
$\pr_i : C \to \PP^1$ as
\[ C \stackrel{\phi_i}{\ra} C_i \stackrel{\xi_i}{\ra} \PP^1 \]
where $\phi_i$ is the quadratic twist map (an isomorphism defined over $L$),
and $\xi_i = (x_i:z_i)$ is the natural double cover.
Let $D_i = \xi_i^*(1:0)$ and $H_i = \phi_i^* D_i = \pr_i^*(1:0)$.
We put $\phi = \phi_1 \phi_2^{-1}$ and $\Xi = \phi_* D_2 - D_1$.
We now prove (i) and (ii).

(i) Let $\iota_1$ and $\iota_2$ be the involutions on $C$ defined
before Lemma~\ref{lem:22}. Since $\sigma(\phi_1) = \phi_1 \iota_1$
and $\sigma(\phi_2) = \phi_2 \iota_2$ it follows that $\sigma(\phi)
\phi^{-1} = \phi_1 \iota_1 \iota_2 \phi_1^{-1}$. Identifying $C$ and
$C_1$ via $\phi_1$, and hence $\Xi$ with $H_2 - H_1$,
it follows by Lemma~\ref{lem:22} that $\sigma(\phi) \phi^{-1}$ is
translation by some $P \in E = \Jac(C_1)$, and the isomorphism
$\Pic^0(C_1) \isom E$ sends $[\Xi] \mapsto -P$. The minus sign
does not matter since $|G|=2$.

(ii) By Lemma~\ref{lem:22}(ii) with $H_1 = \phi_1^* D_1$ and
$H_2 = \phi_2^* D_2$ we have
\begin{align*}
  \divv(f(x_1,z_1;1,0)/z_1^2)
  &= \phi_{1*}(\phi_2^* D_2 + \iota_1^* \phi_2^* D_2 - 2 \phi_1^* D_1) \\
  &= \phi_* D_2 + \sigma(\phi_* D_2) - 2 D_1 \\
  &= N_{L/K} (\Xi). \qedhere
\end{align*}
\end{proof}

\section{Triples of binary quartics and $(2,2,2)$-forms}

\label{sec:222}

Let $E/K$ be an elliptic curve. An $n$-covering of $E$ is a pair
$(C,\nu)$ where $C$ is a smooth curve of genus one, and $\nu : C \to E$
is a morphism, such that, for some choice of isomorphism 
$\psi : C \to E$ defined over $\Kbar$,
there is a commutative diagram
\[ \xymatrix{ C \ar[d]_\psi \ar[dr]^\nu \\ E \ar[r]_{\times n} &  E } \]
The $n$-coverings of $E$ are parametrised by $H^1(K,E[n])$.

Suppose that $C_1$, $C_2$, $C_3$ are $2$-coverings of $E$ 
that sum to zero in $H^1(K,E[2])$. We pick isomorphisms
$\psi_i : C_i \to E$ as above, and let $\eps_i = (\sigma \mapsto
\sigma(\psi_i) \psi_i^{-1})$ be the corresponding cocycle in 
$Z^1(\Gal(\Kbar/K),E[2])$. Our hypothesis is that $\eps_1 + \eps_2 + \eps_3$
is a coboundary. However, by adjusting the choice of $\psi_3$, we may
suppose that $\eps_1 + \eps_2 + \eps_3 = 0$. It may then be checked that
the morphism
\begin{align*}
\mu : C_1 \times C_2 \times C_3 & \to E \\
(P_1,P_2,P_3) & \mapsto \psi_1(P_1) + \psi_2(P_2) + \psi_3(P_3) 
\end{align*}
is defined over $K$.

\begin{Remark} 
\label{rem:many}
We are still free to replace $\psi_3$ 
by $P \mapsto \psi_3(P) + T$ for $T \in E(K)[2]$, and for this reason
there are $\# E(K)[2]$ choices for the map $\mu$.
\end{Remark} 

Suppose further that $C_1$, $C_2$, $C_3$ are defined by binary quartics
$g_1$, $g_2$, $g_3$ with the same invariants $I$ and $J$. Let
$\pi : C_1 \times C_2 \times C_3 \to \PP^1 \times \PP^1 \times \PP^1$
be the map that projects to the $x$-coordinates. Then $S = \pi(\mu^{-1}(0_E))$
is a surface in $\PP^1 \times \PP^1 \times \PP^1$. Geometrically
it is the Kummer surface $(E \times E)/ \{\pm 1\}$.

We write $\disc_k(F)$ for the 
discriminant of a $(2,2,2)$ form $F$ when it is viewed as
a binary quadratic form in the $k$th set of variables.

\begin{Proposition}
\label{prop:222}
The surface $S \subset \PP^1 \times \PP^1 \times \PP^1$ 
is defined by a $(2,2,2)$-form $F$. Moreover we may scale
$F$ so that it has coefficients in $K$, and for all permutations
$i,j,k$ of $1,2,3$ we have $\disc_k(F) = g_i g_j$. 
\end{Proposition}

\begin{proof}
We first consider the special case where $C_1 = C_2 = C_3 = E$.
Suppose that $P_1,P_2,P_3 \in E$ satisfy $P_1 + P_2 + P_3 = 0_E$. If 
we specify the $x$-coordinates of $P_1$ and $P_2$, then in general
this leaves two possibilities for the $x$-coordinate of $P_3$.
The exceptional cases 
are when either $P_1$ or $P_2$ is a $2$-torsion point. 

Let $\Delta_i = \psi_i^{-1} (E[2])$ be the set of ramification points
for $C_i \to \PP^1$. We identify $\Delta_i$ with its image in $\PP^1$,
i.e., the set of roots of $g_i$. The observations in the last paragraph show
that when we project onto the $i$th and $j$th factors, 
$S \to \PP^1 \times \PP^1$
is a double cover ramified over $\Delta_i \times \PP^1$ and 
$\PP^1 \times \Delta_j$. This shows that $S$ is defined by 
a $(2,2,2)$-form $F$. Moreover $\disc_k(F) = \lambda_k g_i g_j$ for 
some $\lambda_1, \lambda_2, \lambda_3 \in K^\times$.
We claim that (i) $\lambda_3 \in (K^\times)^2$ and (ii) $\lambda_1 = 
\lambda_2 = \lambda_3$. It is then clear we may rescale $F$ so that
$\lambda_1 = \lambda_2 = \lambda_3=1$.

(i) Let $C_i$ have equation $y_i^2 = g_i(x_i,z_i)$. 
We note that $K(S) \subset K(C_1 \times C_2)$ is a quadratic extension
of $K(\PP^1 \times \PP^1)$ with Kummer generator
\[  \frac{\disc_3(F)}{z_1^4 z_2^4} = \frac{ \lambda_3  g_1(x_1,z_1) 
g_2(x_2,z_2)}{z_1^4 z_2^4} = \lambda_3 \left( \frac{ y_1 y_2 }
{z_1^2 z_2^2} \right)^2. \]
Since this is a square in $K(C_1 \times C_2)$ it follows
that $\lambda_3 \in (K^\times)^2$.

(ii) Since $g_1,g_2,g_3$ have the same invariants $I$ and $J$, 
we may reduce by the action of 
$\SL_2(\Kbar)\times\SL_2(\Kbar)\times\SL_2(\Kbar)$ to the case 
\[ g_1(x,z) = g_2(x,z) = g_3(x,z) = x^3 z - \tfrac{1}{3} I x z^3 
   - \tfrac{1}{27}J z^4. \]
The result then follows by symmetry.
\end{proof}

\begin{Corollary}
\label{cor}
Let $C_1$, $C_2$, $C_3$ and $F$ be as above. If $a = g_3(1,0) \not= 0$ then
\[ \Psi_{C_1}([C_2]) 
= (K(\sqrt{a})/K,F(x_1,z_1;1,0;1,0)/z_1^2). \]
\end{Corollary}
\begin{proof}
  We put $f(x_1,z_1;x_2,z_2) = F(x_1,z_1;x_2,z_2;1,0)$.
  By Proposition~\ref{prop:222} we have $\disc_1(f) = a g_2(x_2,z_2)$
  and $\disc_2(f) = a g_1(x_1,z_1)$. The curves $C_1$ and $C_2$ are
  therefore isomorphic to those considered in Lemma~\ref{prop:22}.
  Applying Lemma~\ref{prop:22} gives the result.
\end{proof} 

\section{Computing the $(2,2,2)$-forms}

To complete the proof of Theorem~\ref{thm:main} we must 
explain how to compute the $(2,2,2)$-form $F$.
As before it is helpful to first consider the special case
where $C_1 = C_2 = C_3 = E$.

Let $E$ be the elliptic curve $y^2 = x^3 + a x + b$. We consider the maps
\[ \xymatrix{ E \times E \times E \ar[r]^-{\mu} \ar[d]^\pi & E \\
\PP^1 \times \PP^1 \times \PP^1 } \]
where $\mu(P_1,P_2,P_3) = P_1 + P_2 + P_3$ and $\pi$ is the map taking
the $x$-coordinate of each point.
An equation for $S = \pi (\mu^{-1} (0_E))$
is computed as follows. 

Let $P_i=(x_i,y_i)$ for $i=1,2,3$ be points on $E$ with
$P_1+P_2+P_3 = 0_E$. These points lie on a line, say $y = \lambda x + \nu$. 
Then as polynomials in $x$ we have
\[ x^3 + a x + b - (\lambda x + \nu)^2 = (x-x_1)(x-x_2)(x-x_3). \]
Comparing the coefficients of the powers of $x$ we obtain
\begin{align*}
\lambda^2 &= s_1, \\
2 \lambda \nu &= a - s_2, \\
\nu^2 &=  b + s_3, 
\end{align*}
where $s_1, s_2,s_3$ are the elementary symmetric polynomials in 
$x_1,x_2,x_3$. Eliminating $\lambda$ and $\nu$ gives the equation
\[ (a-s_2)^2 - 4 s_1 (b+s_3) = 0.  \]
The required $(2,2,2)$-form $F$ is obtained by homogenising this
equation, i.e. we replace $x_i$ by $x_i/z_i$ and multiply 
through by $z_1^2 z_2^2 z_3^2$.

\begin{Remark}
\label{rem:W}
We have $F(x_1,1;x_2,1;x_3,1) = W_0 x_3^2 - W_1 x_3 + W_2$ where 
\begin{align*}
W_0 & = (x_1 - x_2)^2, \\
W_1 & = 2 (x_1 x_2 + a) (x_1+x_2) + 4b, \\
W_2 & = x_1^2 x_2^2 - 2 a x_1 x_2 - 4 b (x_1+x_2) + a^2.
\end{align*}
These are the formulae used in \cite[Chapter 17]{CaL}
to show that the height on an elliptic curve is a quadratic form.
\end{Remark}

We now turn to the general case. So let 
$S \subset \PP^1 \times \PP^1 \times \PP^1$ 
be as in Section~\ref{sec:222}. 
Let $z(g_i)$ be the cubic invariant, and write $H_1,H_2,H_3$ 
for the binary quadratic forms~\eqref{def:H} over $L = K[\varphi]$ 
associated to $g_1,g_2,g_3$. 

\begin{Theorem}
\label{thm:m}
If $z(g_1) z(g_2) z(g_3) = m^2$ for some $m \in L^\times$, and 
\begin{equation}
\label{get222}
 \frac{H_1 H_2 H_3}{m} = F_0 + F_1 \varphi + F_2 \varphi^2 
\end{equation}
where $F_0,F_1,F_2$ are $(2,2,2)$-forms defined over $K$, 
then $S$ has equation $F_2 = 0$.
\end{Theorem}
\begin{proof}
Let $P_i= (x_i : y_i : z_i) \in C_i$ for $i=1,2,3$, with
$\mu(P_1,P_2,P_3) = 0_E$. Let $Q_i$ be the image of $P_i$ 
under the covering map $C_i \to E$. By the formulae for
the covering map coming from classical invariant theory
(see for example \cite[Proposition 4.2]{Cr}),
the $x$-coordinate of $Q_i$ is
\begin{equation}
\label{eqn:cov}
 \xi_i = \frac{3 h_i(x_i,z_i)}{4 g_i(x_i,z_i)}. 
\end{equation}
We recall from Section~\ref{sec:stat} that 
\begin{equation}
\label{eqn:reln}
 z(g_i)  \frac{ 4 \varphi g_i + h_i}{3} = H_i^2.
\end{equation}
By~\eqref{eqn:cov}, \eqref{eqn:reln} and the equation 
$y_i^2 = g_i(x_i,z_i)$ for $C_i$ we have 
\[ \xi_i + 3 \varphi = \frac{ 9 H_i^2 }{4 z(g_i) y_i^2} \]
and hence
\begin{equation}
\label{eq1}
 \prod_{i=1}^3 (\xi_i + 3 \varphi) = 
\left( \frac{27 H_1 H_2 H_3}{8 m y_1 y_2 y_3} \right)^2.
\end{equation}
Since $Q_1 + Q_2 + Q_3 = 0_E$ these points lie on a line,
say $y = \lambda x + \nu$ for some $\lambda,\nu \in K$. Then 
as a polynomial in $x$ we have
\[  x^3 - 27 I x - 27 J - ( \lambda x + \nu)^2 = 
(x-\xi_1)(x-\xi_2)(x-\xi_3). \]
Putting $x = -3 \varphi$ gives
\begin{equation}
\label{eq2}
 \prod_{i=1}^3 (\xi_i + 3 \varphi) = (\nu - 3 \lambda \varphi)^2.
\end{equation}

We first suppose $E(K)[2] =0$. In this case $L$ is a field, so
comparing~\eqref{eq1} and~\eqref{eq2} we have
\[\frac{27 H_1 H_2 H_3}{8 m y_1 y_2 y_3}  = \pm (\nu - 3 \lambda \varphi),\]
in $L(S)$. Taking the coefficient of $\varphi^2$ 
we see that $F_2$ vanishes on $S$. In general there are $\#E(K)[2]$
choices for the square root, up to sign, and these correspond to 
the $\#E(K)[2]$ choices in Remark~\ref{rem:many}.
 
It remains to check that $F_2$ is not identically zero. 
For this we may work over an algebraically closed field.
Then by a change of coordinates we may suppose that $g_i$ and $h_i$ 
are linear combinations of $x_i^4 + z_i^4$ and $x_i^2 z_i^2$.
The singular quartics in this pencil are
$(x_i^2 - z_i^2)^2$, $(x_i^2 + z_i^2)^2$ and $(x_i z_i)^2$.
Since $L \isom K \times K \times K$ we may identify $H_i$ as a triple 
of binary quadratic forms. These are 
non-zero multiples of $x_i^2 - z_i^2$, $x_i^2 + z_i^2$ and $x_i z_i$,
in this order if we made a suitable change of coordinates.
(This last claim may be checked without any calculation if we use
stereographic projection to identify the roots of the binary quadratic
forms with the vertices of an octahedron, and then rotate the octahedron.)
Therefore the space of $(2,2,2)$-forms spanned by $F_0,F_1,F_2$
contains the forms 
\[ (x_1^2 - z_1^2)(x_2^2 - z_2^2)(x_3^2 - z_3^2), \quad
(x_1^2 + z_1^2)(x_2^2 + z_2^2)(x_3^2 + z_3^2), \quad
x_1 z_1 x_2 z_2 x_3 z_3.  \]
Since these are linearly independent, it follows that $F_2$ is non-zero.
\end{proof}

\begin{ProofOf}{Theorem~\ref{thm:main}}
Let $F = F_2$ be the equation for $S$ in Theorem~\ref{thm:m}.
We specialise the last two sets of variables in~\eqref{get222} 
to $(1,0)$. Then comparing with~\eqref{getgamma} we have
$F(x,z;1,0;1,0) = \gamma_1(x,z)$. By Corollary~\ref{cor}
we have \[ \Psi_{C_1}([C_2]) = (K(\sqrt{a})/K,\gamma_1(x,z)/z^2), \]
where $a = g_3(1,0)$. Then by \eqref{def:ctp} we have
\[ \langle [C_1],[C_2] \rangle_{\CT} = \sum_{v \in M_K} 
\inv_v (K_v(\sqrt{a})/K_v,\gamma_1(x_v,z_v)/z_v^2). \]
Subject to identifying $\mu_2 = \frac{1}{2}\Z/\Z$,
the Hilbert norm residue symbol is given by
   \[ (a,b)_v = \inv_v (K_v(\sqrt{a})/K_v, b). \]
This gives the formula in Theorem~\ref{thm:main}, except that 
we have $g_3(1,0)$ in place of $g_2(1,0)$. As noted in 
Remark~\ref{rems}(v), this change does not matter.
\end{ProofOf}

\begin{Remark}
\label{rem:gamma}
To show that $\gamma_1(x,z)$ is not identically zero we show
more generally that $F$ cannot be made to vanish by specialising
two of the sets of variables. Indeed, by considering $F$ as
given in Remark~\ref{rem:W}, it suffices to show that the polynomials
$W_0,W_1,W_2$ never simultaneously vanish. 
This may be checked by setting $x_1 = x_2 = x$ and computing that
the resultant of $W_1$ and $W_2$ is $2^8(4 a^3 + 27 b^2)^2$. This last
expression is non-zero, by definition of an elliptic curve.
\end{Remark}


\begin{thebibliography}{MM}
\frenchspacing
\renewcommand{\baselinestretch}{1}

\bibitem{BH}
M. Bhargava and W. Ho,
Coregular spaces and genus one curves,
{\em Camb. J. Math.} {\bf{4}} (2016), no. 1, 1--119.

\bibitem{BSDI}
B.J. Birch and H.P.F. Swinnerton-Dyer, 
Notes on elliptic curves, I.
{\em J. reine angew. Math.} {\bf{212}} (1963) 7--25.

\bibitem{magma}
W. Bosma, J. Cannon, and C. Playoust,
The Magma algebra system, I. The user language,
{\em J. Symbolic Comput.}, {\bf{24}} (1997), 235--265. 

\bibitem{CaIV}
J.W.S. Cassels, 
Arithmetic on curves of genus 1, 
IV. Proof of the Hauptvermutung,
{\em J. reine angew. Math.} {\bf{211}} (1962) 95--112.

\bibitem{CaL}
J.W.S. Cassels, 
{\em Lectures on elliptic curves},
LMS Student Texts, {\bf{24}}, Cambridge University Press, Cambridge, 1991. 

\bibitem{Ca98}
J.W.S. Cassels,
Second descents for elliptic curves,
{\em J. reine angew. Math.} {\bf{494}} (1998), 101--127.

\bibitem{CrTables}
J.E. Cremona, 
{\em Algorithms for modular elliptic curves},
Second edition, Cambridge University Press, Cambridge, 1997. 

\bibitem{Cr}
J.E. Cremona, 
Classical invariants and 2-descent on elliptic curves,
{\em J. Symbolic Comput.} {\bf{31}} (2001), no. 1-2, 71--87.

\bibitem{CF}
J.E. Cremona and T.A. Fisher, 
On the equivalence of binary quartics,
{\em J. Symbolic Comput.} {\bf{44}} (2009), no. 6, 673--682.

\bibitem{D}
S. Donnelly, {\em Algorithms for the Cassels-Tate pairing}, preprint, 2015.

\bibitem{F}
T.A. Fisher, 
The Cassels-Tate pairing and the Platonic solids,
{\em J. Number Theory} {\bf{98}} (2003), no. 1, 105--155.

\bibitem{FSS}
T.A. Fisher, E.F. Schaefer and M. Stoll,
The yoga of the Cassels-Tate pairing,
{\em LMS J. Comput. Math.} {\bf{13}} (2010), 451--460.

\bibitem{GH}
P. Griffiths and J. Harris, 
On Cayley's explicit solution to Poncelet's porism,
{\em  Enseign. Math.} (2) {\bf{24}} (1978), no. 1-2, 31--40. 

\bibitem{Lichtenbaum}
S. Lichtenbaum, 
Duality theorems for curves over $p$-adic fields.
{\em Invent. Math.} {\bf{7}} (1969) 120--136.

\bibitem{pari}
The PARI~Group, PARI/GP version \texttt{2.13}, Univ. Bordeaux, 2022,
\url{http://pari.math.u-bordeaux.fr/}

\bibitem{PS}
B. Poonen and M. Stoll, 
The Cassels-Tate pairing on polarized abelian varieties,
{\em Ann. of Math.} (2) {\bf{150}} (1999), no. 3, 1109--1149.

\bibitem{Weil54}
A. Weil, 
Remarques sur un m\'emoire d'Hermite,
{\em Arch. Math. (Basel)} {\bf{5}}, (1954) 197--202.

\bibitem{Yan}
J. Yan, {\em Computing the Cassels-Tate pairing for Jacobian varieties of genus two curves}, PhD thesis, University of Cambridge, 2021,
\url{https://doi.org/10.17863/CAM.72729}

\end{thebibliography}
\end{document}